\input ./style/arxiv-general.cfg
\documentclass[aap,MSNbibl,seceqn,dvips]{arximspdf}
\makeatletter
   \@ifpackageloaded{graphicx}{}{\usepackage{graphicx}}
\makeatother

\doi{10.1214/15-AAP1144}
\volume{26}
\issue{4}
\pubyear{2016}
\firstpage{2193}
\lastpage{2210}
\docsubty{FLA}

\makeatletter
\def\mid{|}
\def\mathbbm{\mathbb}
\newcommand{\eqref}[1]{(\ref{#1})}
\def\cal{\mathcal}
\newtheorem{theorem}{Theorem}[section]
\newtheorem{proposition}[theorem]{Proposition}
\newtheorem{lemma}[theorem]{Lemma}
\newtheorem{corollary}[theorem]{Corollary}

\newproclaim{definition}{Definition}[section]
\newproclaim{example}[definition]{Example}
\newproclaim{remark}[definition]{Remark}
\newproclaim{case}[definition]{Case}
\newproclaim{condition}[definition]{Condition}

\newcommand{\toinf}{\to\infty}
\newcommand{\tozero}{\to0}

\newcommand{\tind}{\mathop{\mathrm{inj}}}
\newcommand{\lito}{\mathrm{o}}

\def\sub{\mathrm{sub}}
\def\spn{\mathop{\mathrm{span}}}
\newcommand{\eps}{\varepsilon}

\newcommand{\IE}{\mathbbm{E}}
\newcommand{\IP}{\mathbbm{P}}
\newcommand{\Vol}{\mathop{\mathrm{Vol}}}

\def\cF{{\cal F}}
\def\cC{{\cal C}}
\def\cP{{\cal P}}
\def\cX{{\cal X}}
\def\I{\mathrm{I}}

\makeatother

\begin{document}
\begin{frontmatter}

\title{Dense graph limits under respondent-driven sampling}
\runtitle{Dense graph limits under respondent-driven sampling}

\begin{aug}
\author[A]{\fnms{Siva} \snm{Athreya}\corref{}\thanksref{t1,t2}\ead[label=e1]{athreya@isibang.ac.in}}
\and
\author[B]{\fnms{Adrian} \snm{R\"ollin}\thanksref{t2}\ead[label=e2]{adrian.roellin@nus.edu.sg}}
\thankstext{t1}{Supported in part by CPDA-grant from the
Indian Statistical Institute, Bangalore.}
\thankstext{t2}{Supported in part by NUS Research Grant R-155-000-124-112.}
\affiliation{Indian Statistical Institute and National University of Singapore}
\runauthor{S. Athreya and A. R\"ollin}
\address[A]{Statmath Unit\\
Indian Statistical Institute\\
8th Mile Mysore Road\\
Bangalore, 560059\\
India\\
\printead{e1}}
\address[B]{Department of Statistics and Applied Probability\\
National University of Singapore\\
6 Science Drive 2\\
Singapore 117546\\
\printead{e2}}
\end{aug}

%
\received{\smonth{4} \syear{2014}}
%
\revised{\smonth{9} \syear{2015}}

%
\begin{abstract}
We consider certain respondent-driven sampling procedures on dense
graphs. We show that if the sequence of the vertex-sets is ergodic
then the limiting graph can be expressed in terms of the original
dense graph via a transformation related to the invariant measure of
the ergodic sequence. For specific sampling procedures, we describe the
transformation explicitly.
\end{abstract}

%

%
\end{frontmatter}

\section{Introduction}\label{sec1}

Respondent-driven sampling (RDS) of social networks has received a lot
of attention since \cite{Heckathorn1997} and \cite{Heckathorn2002},
and many studies have implemented the procedure in order to obtain
estimates about properties of so-called ``hidden'' or
``hard-to-reach'' populations. The basic idea is to start with a
convenience sample of participants, to ask the participants for
\emph{referrals} among their peers and then to iterate this
process. It is intuitively clear that one cannot hope to obtain an
unbiased sample in this manner as individuals with higher connectivity
are more likely to appear in the sample than individuals with lower
connectivity. In order to avoid this bias, one of the key assumptions
of \cite{Heckathorn1997} is that each individual in the network has
the same degree. Subsequent refinements of the procedure have been
proposed to overcome such restrictions; see \cite{Volz2008}.

Respondent-driven sampling has also received quite some criticism. Besides
inadequate control of biases for finite samples, another major issue
can be
the underestimation of sample variance; see, for example, \cite
{Gile2010} and
\cite{Goel2010}.

The main purpose of this article is to take a first (and very
preliminary) step in establishing a rigorous theory of RDS on dense
graphs in order to understand the graphs produced under various
sampling procedures. Our main contribution is that the limit of a
dense graph sequence obtained through a specific respondent-driven
sampling procedure, where the sequence of the vertex-sets is ergodic,
can be expressed in terms of the original graph limit and a
transformation related to the invariant measure of the ergodic
sequence. The transformation, in essence, confirms the bias toward
nodes with larger degrees.


In practice, researchers typically are interested in estimating
certain quantities at population or subpopulation level, such as
prevalence of STIs, sexual contact frequencies, condom use, etc. Hence,
for each node in the network, additional data is collected, and the
main question of RDS becomes how to obtain representative estimates of
those quantities from the RDS sample. In this article, we will only be
interested in the network itself and the question how specific RDS
procedures bias the network. However, if, for example, a quantity of
interest (such as STI prevalence) correlates with the degree that a
node has in that network, then it is obviously important to understand
the bias in the network itself in order to understand the resulting
bias of that quantity of interest.

It is also important to note that the sampling procedure analysed in
this article is not representative for what is mostly being done in
practice. In particular, we assume that after the referral chain has
been sampled (or rather ``revealed''), all yet unknown connections
between the subjects in the sample are also revealed. In other words,
if Subject A refers to Subject B and Subject B refers to Subject C, we
assume that, in a second step, the relationship between Subjects A
and C be revealed, also. In practice that last relationship typically
remains unknown, unless either A refers to C or C refers to A.

Our proof is based on subgraph counts convergence and ergodicity of the
sampling procedure.
Subgraph counts can be written as incomplete $U$-statistics or
generalised $U$-statistics,
but there does not seem to exist a well-established general theory
that would cover ergodic sequences in the generality needed in this
article. However, noticing that, in our model, the conditional
expectation of a subgraph count, conditioned on the vertex set, is a
\emph{complete} $U$-statistic, we can resort to the well-established
theory of $U$-statistics, in particular for ergodic sequences. We
modify the arguments of \cite{Aaronson1996} in order to deal with
nonstationary sequences, which seems a more realistic assumption in
the context of RDS.

The rest of the paper is organised in the following manner. We
conclude this section with a focussed review of the dense graph
literature. We state the model and main results in Section~\ref{sec2} and prove
them in Section~\ref{sec3}. We finally discuss some applications in Section~\ref{sec4}.

\subsection{A brief introduction to dense graphs}\label{sec1.1}

Dense graph theory has been introduced by
\cite{Lovasz2006}. Diaconis and Janson \cite{Diaconis2008} made connections with
earlier work of \cite{Aldous1981}. Let us briefly summarise those
parts of dense graph theory which are needed in this paper; see the
monograph \cite{Lovasz2012} and \cite{Borgs2008,Borgs2012} for an
in-depth discussion, or \cite{Bollobas2009} for another
introduction with extensions to sparse graphs.

Let $(G_n)_{n\geq1}$ be a sequence of graphs, where for simplicity we
assume that the vertex set of graph $G_n$ is $\{1,\ldots,n\}$. Assume
further that the number of edges $E(G_n)$ in $G_n$ is of order $n^2$;
that is, $\lim\inf_{n\toinf}  (E(G_n) n^{-2} )>0$. We
call $(G_n)_{n\geq1}$ a \emph{dense graph sequence}.

\textit{Subgraph distance}. Let $K_n$ be the complete graph on
$\{1,\ldots,n\}$. For any (small) graph $F$ on $k$ vertices, let
$X_F(G)$ be the number of copies of $F$ in a (large) graph $G$ on $n$
vertices. Define the normalised subgraph count
\[
t(F,G) = \frac{\vert\tind(F,G)\vert }{(n)_k} = \frac{ X_F(G)}{X_F(K_n)},
\]
where $\tind(F,G)$ denotes the set of injective graph homomorphisms
of $F$ into $G$, that is, the functions that map the vertices of $F$
into the vertices of $G$ injectively such that connected vertices
in $F$ remain connected in $G$. Here, as usual, $(n)_k =
n(n-1)\cdots(n-k+1)$. When $\mid F
\mid> \mid G \mid$, we define $t(F,G) = 0$.

Note that $0\leq t(F,G)\leq1$. Let ${\cal F}$ denote the class of
isomorphism classes on finite graphs and let  $(F_{i})_{i\geq1}$ be a
particular enumeration of $\cF$, where each $F_i$ is the
representative of an isomorphism class. We can define a distance
function between graphs by
\[
d_{\sub} \bigl(G,G' \bigr) = \sum
_{i\geq1} 2^{-i} \bigl\vert t(F_i,G) - t
\bigl(F_i,G' \bigr) \bigr\vert.
\]
A key feature of $d_{\sub}$ is that there is a
natural completion of $({\cal F}, d_{\sub})$ by standard
kernels. We call any function $\kappa:[0,1]^2\to[0,1]$ that is
measurable and symmetric a \emph{standard kernel}. For $F$ a graph
on $k$ vertices, we can extend the definition of $t(F, G)$ to kernels
by means of
\[
t(F,\kappa) = \int_{[0,1]^k} \prod_{\{i,j\}\in E(F)}
\kappa(x_i,x_j) \,dx_1\cdots
dx_k,
\]
where $E(F)$ is the set of edges in $F$. One of the key results in
dense graph theory is the following theorem.



\begin{theorem}\label{THM0} Let $(G_n)_{n\geq1}$ be a dense graph
sequence which is Cauchy with respect to $d_{\sub}$. Then there exists a
standard kernel $\kappa$ such that
%
\begin{equation}
\label{1} d_{\sub}(G_n,\kappa) \rightarrow0
\end{equation}
as $n \rightarrow\infty$.
\end{theorem}

For a proof of the above (see \cite{Lovasz2006}), which uses Szemer\'edi
partitions and the Martingale convergence theorem, or
\cite{Diaconis2008}, who show that it can be proved using results from
\cite{Hoover1979} and \cite{Aldous1981}. Note that $\kappa$ above is
in general not unique, but this will not be of importance for what
follows; we refer to \cite{Borgs2008,Borgs2012} for a discussion of
this and related questions. We refer to \cite{Lovasz2012}, Chapter~11,
for a detailed discussion of convergence of dense graph sequences.

\section{Model and main results}\label{sec2}
A convenient way of ``creating'' finite (random) graphs on $n$
vertices from a standard kernel $\kappa$ is the following model, which
we will denote by $G(n,\kappa)$. Firstly, let $U_1,\ldots,U_n$ be
i.i.d. with uniform distribution on $[0,1]$. Second, for each two
vertices $i$ and $j$, connect them with probability $\kappa(U_i,U_j)$,
independently of all the other edges. It is not difficult to prove
that
%
\begin{equation}
\label{1a} d_{\sub} \bigl(G(n,\kappa),\kappa \bigr) \tozero
\end{equation}
almost
surely as $n\toinf$. This is, in some sense, the basic law of large
numbers in dense graph theory. In this article, instead of sampling the
labels i.i.d. and uniformly from $[0,1]$, we will allow the labels to
be sampled in a more general way.

\textit{The random graph ${G(x,\kappa)}$}.
Let $x=(x_1,\ldots,x_n) \in[0,1]^n$ be
fixed. Define the random graph $G(x, \kappa)$ by connecting vertices
$i$ and $j$ with probability $\kappa(x_i,x_j)$ independently of all
other vertices. Clearly, $G ((U_1,\ldots,U_n),\kappa )$ is equivalent
to $G(n,\kappa)$. We will show a version of \eqref{1a} for $G(X,
\kappa)$, where---in essence---the labels $X$ are allowed to come
from a general ergodic sequence.

To this end, let $\kappa$ be a standard kernel and let $g:[0,1]\to
[0,1]$ be a
Lebesgue-measurable function. Define the
$g$-transformed kernel
\[
\kappa_g(x,y) = \kappa \bigl(g(x),g(y) \bigr).
\]

\begin{theorem}\label{THM1} Let $X^{(n)} = (X_{n,1}, \ldots, X_{n,i},
\ldots, X_{n,n})$, $n\geq1$, be a triangular array of random variables
taking values
in $[0,1]$. Assume that there is a probability measure
$\pi$ on $[0,1]$ such that the following two conditions hold:
\begin{longlist}[(ii)]
\item[(i)]for all bounded and measurable
functions $f$, we have
%
\begin{equation}
\label{3} \lim_{n\toinf}\frac{1}{n}\sum
_{i=1}^n f(X_{n,i}) = \int
_0^1 f(x) \,d\pi(x)
\end{equation}
almost surely;
\item[(ii)] $\kappa(\cdot, \cdot)$ is continuous $(\pi\times\pi)$-almost
everywhere.
\end{longlist}
Then
%
\begin{equation}
\label{3a} d_{\sub} \bigl(G \bigl(X^{(n)},\kappa \bigr),
\kappa_{\tau^{-1}} \bigr)\tozero
\end{equation}
almost surely, where $\tau(x)=\pi([0,x])$ is the distribution function
of $\pi$ and
\[
\tau^{-1}(v) = \inf \bigl\{u\in[0,1]: \tau(u)\geq v \bigr\}
\]
its generalised inverse.
\end{theorem}

\subsection{Respondent driven sampling}\label{sec2.1} The way we think
about respondent-driven sampling in this article is by means of the
following two-step procedure. First, sample a set of
subjects $X^{(n)}=(X_{n,1},\ldots,X_{n,n})$, where new subjects are
added by referrals; each subject $i$ obtains a unique
label $X_{n,i}\in[0,1]$ (note that the $X_{n,i}$ are just the labels
of the nodes, not any additional observation related to that
node). If $i$ referred to $j$, or $j$ referred to $i$, an edge between
the two nodes is added; denote the resulting graph by $H_n$. Second,
the remaining relationships are then revealed by connecting $i$
and $j$ with probability $\kappa(X_{n,i},X_{n,j})$, unless they are
already connected in $H_n$. We define the model precisely below.

\textit{The random graph ${G(x,H_n,\kappa)}$}.
Let $x=(x_1,\ldots,x_n) \in[0,1]^n$ be fixed and let $H_n$ be a given
graph on the vertices $\{1,\ldots,n\}$. Define the random
graph $G(x,H_n,\kappa)$ on the same set of vertices as follows:
\begin{itemize}
\item if there is an edge between $i$ and $j$ in $H_n$, then connect
vertices $i$ and $j$ in $G(x,H_n,\kappa)$;
\item if there is no edge between $i$ and $j$ in $H_n$, then
connect $i$ and $j$ in $G(x,H_n,\kappa)$ with
probability $\kappa(x_i,x_j)$ independently of all other vertices.
\end{itemize}

\begin{corollary} \label{cor1} Let $X^{(n)}= (X_{n,i})_{1\leq i\leq
n}$, $\kappa$ and $\tau$ be as in Theorem~\ref{THM1}, satisfying
conditions \textup{(i)} and \textup{(ii)}. If the number of edges in $H_n$ is $\lito
(n^2)$, then
%
\begin{equation}
\label{3b} d_{\sub} \bigl(G \bigl(X^{(n)},\kappa,
H_n \bigr),\kappa_{\tau^{-1}} \bigr)\tozero
\end{equation}
almost surely.
\end{corollary}

The above corollary is an easy consequence of Theorem~\ref{THM1} and
the counting lemma
\cite{Lovasz2012}, Lemma~10.22. For completeness sake, we present a
proof in the
next section.

\subsection{Remarks}\label{sec2.2}
Before concluding this section, we discuss some interesting aspects
around Theorem~\ref{THM1}.

\def\cA{\mathcal{A}}

\textit{Reference measure space}. Using $[0,1]$ and the Lebesgue
measure as reference is only a matter of convenience and in line with
the prevailing literature. However, in order to shed some light on the
main result, let us state Theorem~\ref{THM1} in greater generality; we
refer to \cite{Lovasz2012}, Chapter~13, for a more in-depth discussion
of this setting.

Let $(\cX, \cA, \mu)$ be a probability space, and let $\kappa:\cX
\times
\cX\to[0,1]$ be a symmetric and $(\cA\times\cA)$-measurable function.
For any graph $F$ on $k$ vertices, where $k\geq1$, we can easily
generalise the definition of the subgraph density to
\[
t_\mu(F,\kappa) = \int_{\cX^k} \prod
_{\{i,j\}\in E(F)}\kappa(x_i,x_j) \,d
\mu(x_1)\cdots \,d\mu(x_k).
\]
Moreover, for $U_1,\ldots,U_n$ being i.i.d. random elements taking
values in $\cX$ with common distribution $\mu$, the random graph
model $G((U_1,\ldots,U_n),\kappa)$ can be defined in a
straightforward manner, and one can prove that
\[
d_{\mu,\sub} \bigl(G \bigl((U_1,\ldots,U_n),\kappa
\bigr),\kappa \bigr) \tozero,\qquad n\toinf.
\]
It is important to emphasise that $t_\mu$ and, as a result, the metric
$d_{\mu,\sub}$ depend on the reference measure $\mu$.

Now, assume $(X_{n,i})_{1\leq i\leq n}$ is a triangular array of $\cX
$-valued random elements such that
\[
\lim_{ n \rightarrow\infty} \frac{1}{n}\sum_{i=1}^n
f(X_{n,i}) = \int_{\cX} f(x) \,d\pi(x)
\]
for some probability measure $\pi$ on $(\cX,\cA)$. Now, assume $\cX
$ is
a Polish space. If there is a function $g:\cX\to\cX$ such that
\[
g(U_1)\sim\pi,
\]
and if $\kappa$ is continuous $(\pi\times\pi)$-almost everywhere, then
%
\begin{equation}
\label{3c} d_{\mu,\sub} \bigl(G \bigl(X^{(n)},\kappa \bigr),
\kappa_g \bigr) \tozero, \qquad n\toinf.
\end{equation}
In the case where $(\Omega, \cA, \mu)$ is the interval $[0,1]$ and
$\mu
$ the uniform distribution, $g$ can be identified as the generalised
inverse of the distribution function of $\pi$, but note that, for
general spaces $\cX$, it is difficult to find such $g$ explicitly.

It is illuminating to consider the following alternative way to state
\eqref{3c}. From the proof of Theorem~\ref{THM1} [cf. \eqref
{kttf}], it
becomes clear that, by changing the reference measure from $\mu$ to
$\pi
$, \eqref{3c} can also be written as
%
\begin{equation}
\label{3d} d_{\pi,\sub} \bigl(G \bigl(X^{(n)},\kappa \bigr),
\kappa \bigr) \tozero, \qquad n\toinf.
\end{equation}
Although \eqref{3c} and \eqref{3d} are equivalent, the former
statement is
more important in the context of RDS, since we are interested in
describing the distortion of the network through biased sampling.

\textit{Necessity of condition} (ii). Condition (ii)
in Theorem~\ref{THM1} can be replaced by other conditions, but that it
cannot be dispensed with entirely can be seen from \cite{Aaronson1996}, Example~4.1. We state the example below with notation as
applicable to our case.

Consider the interval $(0,1)$ and define the mapping $\phi: (0,1) \to
(0,1)$ as
\[
\phi(x) = 2x\ (\mbox{mod $1$}).
\]
Let $X_1\in(0,1)$ fixed, and let $X_n = \phi^{n}(X_1)$. It follows
from standard ergodic theory that $X_1,X_2,\ldots$ is ergodic with the
Lebesgue measure as its invariant measure, that is,
\[
\frac{1}{n}\sum_{i=1}^n
f(X_i) \to\int_0^1 f(x) \,dx.
\]
Define the set
\[
L = \bigl\{ (x_1,x_2) \in(0,1)^2 :
\mbox{$x_1 \in(0,1)$ and $x_2 = \phi ^{n}(x_1)$
for some $n \geq1$} \bigr\}.
\]
So, \eqref{3} is satisfied with $\pi$ being the Lebesgue measure; hence
$\tau(x) =x$ and $\tau^{-1}(x)=x$. Define the standard kernel $\kappa
(x,y) = \I[\mbox{$(x,y)\in L$ or $(y,x)\in L$}]$. Now, on the one hand
we have
\[
\frac{\sum_{i=1}^n\kappa(X_i, X_j)}{n(n-1)} = 1\qquad \mbox{for all $n \geq1$}.
\]
On the other hand,
\[
\int\kappa(x, y)\,dx \,dy = 0
\]
since $L$ is the countable union of null sets with respect to the
two-dimensional Lebesgue measure. Thus,
\[
\lim_{n\toinf} t \bigl(F,G \bigl((X_1,
\ldots,X_n),\kappa \bigr) \bigr) = 1 \neq0 = t(F,\kappa).
\]
Since $L$ is dense in $[0,1]$, the standard kernel $\kappa$ is nowhere
continuous and does therefore not satisfy condition (ii).

\section{Proof of Theorem \texorpdfstring{\protect\ref{THM1}}{2.1}}\label{sec3}

We will need a law of large numbers of a particular $U$-statistic for
the proof of Theorem~\ref{THM1}. This essentially allows us to go from
a simple ergodic theorem to a higher order ergodic theorem. Toward
that we define
%
\begin{equation}
\mu_F(x) = \frac{1}{(n)_k}\sum_{(i_1,\ldots,i_k)}
T_F(x_{i_1},\ldots,x_{i_k})
\end{equation}
with
%
\begin{equation}
\label{2} T_F(z_1,\ldots,z_k) = \prod
_{\{i,j\}\in F} \kappa(z_i,z_j),
\end{equation}
for $x=(x_1,\ldots,x_n) \in[0,1]^n$ and $z_i \in[0,1], 1 \leq i \leq
k$. Here, the summation $\sum_{(i_1,\ldots,i_k)}$ ranges over all
vectors $(i_1,\ldots,i_k)$ with mutually different coordinates.

The following result was proved by \cite{Aaronson1996} for ergodic
stationary sequences, but we note that the key assumption is \eqref
{3}, so
that their proof generalises to nonstationary triangular arrays, which
is more appropriate for the applications we have in mind.

\begin{lemma}\label{lem3} Let $X^{(n)} = (X_{n,i})_{1\leq i \leq n}$
be a
triangular array of random variables taking values in $[0,1]$ and
satisfying \eqref{3} almost surely. Then, for any fixed graph $F$ of size $k$,
\[
\lim_{n\toinf} \mu_F \bigl(X^{(n)} \bigr)
\to \IE T_F(V_1,\ldots,V_k)
\]
almost surely, where $V_1,\ldots,V_k$ are i.i.d. random
variables with distribution $\pi$.
\end{lemma}

\begin{pf} Our proof is a close imitation of the proof of \cite{Aaronson1996}, Theorem
U. Denote by $\cC_k$ the set of all functions
from $[0,1]^k$ to $[0,1]$,
continuous $\pi^{(k)}$-almost everywhere, where $\pi^{(k)}=\bigotimes_{i=1}^k\pi$. For $x=(x_1,\ldots,x_n) \in
[0,1]^n$ and $h \in\cC_k$, define
\[
\nu_{h}(x) = \frac{1}{n^{k}}\sum_{ 1\leq i_1,\ldots,i_k\leq n}
h(x_{i_1},\ldots,x_{i_k}).
\]
Let
\[
\cP_k = \Biggl\{ h \in\cC_k : \exists
h_1,\ldots,h_k\in\cC_1\mbox{ such that }h(x) =
\prod_{i=1}^{k}h_{i}(x_{i})
 \Biggr\}.
\]
For any $h \in\cP_k$, we have
\[
\nu_{h} \bigl(X^{(n)} \bigr) = \prod
_{i=1}^{k} \Biggl(\frac{1}{n}\sum
_{j=1}^{n}h_{i}(X_{n,j}) \Biggr).
\]
So, by \eqref{3},
\[
\nu_{h} \bigl(X^{(n)} \bigr) \to\prod
_{i=1}^{k} \int_{[0,1]}
h_{i}(x_{i}) \,d \pi(x_{i}) = \IE h
\bigl(V^{(k)} \bigr)
\]
almost surely as $n\toinf$, where we set $V^{(k)}=(V_1,\ldots,V_k)$ to
shorten formulas. It is easily seen that the above holds whenever $h\in
\spn(\cP_k)$.

Fix $\eps>0$ and let $h \in\cC_k$. As $h$ is continuous $\pi
^{(k)}$-almost everywhere and $\pi^{(k)}$-integrable, there exist $s_1,
s_2 \in\spn(\cP_k)$ such that
\[
\textup{(a)}\quad \vert h - s_1\vert \leq s_2\qquad \mbox{$
\pi^{(k)}$-almost everywhere,}\qquad \textup{(b)}\quad\int s_2 \,d
\pi^{(k)} \leq\eps.
\]
As $s_1,s_2 \in\spn(\cP_k)$, there exists (random) $N$ such that, for
$n \geq N$,
\[
\bigl\vert\nu_{s_1} \bigl(X^{(n)} \bigr) - \IE s_1
\bigl(V^{(k)} \bigr)\bigr\vert \leq\eps
\]
and
\[
\nu_{s_2} \bigl(X^{(n)} \bigr) \leq\IE s_2
\bigl(V^{(k)} \bigr) + \eps\leq2 \eps.
\]
Hence, for $n \geq N$,
\begin{eqnarray*}
&&\bigl\vert\nu_{h} \bigl(X^{(n)} \bigr) - \IE h
\bigl(V^{(k)} \bigr)\bigr\vert
\\
&&\qquad \leq\bigl\vert \nu_{h} \bigl(X^{(n)} \bigr) -
\nu_{s_1} \bigl(X^{(n)} \bigr) \bigr\vert +\bigl\vert \nu_{s_1}
\bigl(X^{(n)} \bigr) - \IE s_1 \bigl(V^{(k)} \bigr)
\bigr\vert \\
&&\qquad\quad{}+ \bigl\vert \IE s_1 \bigl(V^{(k)} \bigr) - \IE h
\bigl(V^{(k)} \bigr) \bigr\vert
\\
&&\qquad \leq\nu_{\vert h -s_1\vert } \bigl(X^{(n)} \bigr) + \bigl\vert
\nu_{s_1} \bigl(X^{(n)} \bigr) - \IE s_1
\bigl(V^{(k)} \bigr) \bigr\vert + \IE s_2 \bigl(V^{(k)}
\bigr)
\\
&&\qquad \leq\nu_{s_2} \bigl(X^{(n)} \bigr) +\bigl \vert
\nu_{s_1} \bigl(X^{(n)} \bigr) - \IE s_1
\bigl(V^{(k)} \bigr) \bigr\vert + \IE s_2 \bigl(V^{(k)}
\bigr)
\\
&&\qquad \leq4\eps.
\end{eqnarray*}
Thus, for all $h\in\cC_k$,
%
\begin{equation}
\label{4} \lim_{n\toinf} \nu_{h}
\bigl(X^{(n)} \bigr) \to\IE h \bigl(V^{(k)} \bigr)
\end{equation}
almost surely. Let $T_F$ be as in \eqref{2}. As $T_F$ is bounded by $1$,
we see that there exists $c_1 >0$ such that
\begin{eqnarray*}
&&\bigl\vert \mu_F \bigl(X^{(n)} \bigr) - \nu_{T_F}
\bigl(X^{(n)} \bigr)\bigr\vert
\\
&&\qquad = \biggl\vert \frac
{1}{(n)_k}\sum_{(i_1,\ldots,i_k)}
T_F(X_{i_1},\ldots,X_{i_k}) - \frac{1}{n^{k}}
\sum_{ 1\leq i_1,\ldots
,i_k\leq n} T_F(X_{i_1},
\ldots,X_{i_k}) \biggr\vert
\\
&&\qquad\leq\frac{c_1}{n}.
\end{eqnarray*}
As $T_F\in\cC_k$, the result follows.
\end{pf}

\begin{pf*}{Proof of Theorem~\ref{THM1}}
It is enough to show that, for every
graph $F$,
%
\begin{equation}
\label{5} \lim_{n\toinf} \bigl\vert t \bigl(F,G
\bigl(X^{(n)}, \kappa \bigr) \bigr) - t(F,\kappa_{\tau^{-1}})\bigr\vert = 0
\end{equation}
almost surely. Using the triangle inequality,
%
\begin{eqnarray}
\label{tistep} && \bigl\vert t \bigl(F,G \bigl(X^{(n)},\kappa \bigr) \bigr) -
t(F, \kappa_{\tau^{-1}})\bigr\vert
\nonumber
\\[-8pt]
\\[-8pt]
\nonumber
& &\qquad \leq\bigl\vert t \bigl(F,G \bigl(X^{(n)},\kappa \bigr) \bigr) -
\mu_F \bigl(X^{(n)} \bigr)\bigr\vert + \bigl\vert \mu _F
\bigl(X^{(n)} \bigr) - t(F,\kappa_{\tau^{-1}})\bigr\vert.
\end{eqnarray}
By definition of $\tau^{-1}$, $V_i$ has the same distribution as $\tau
^{-1}(U_i)$, so that
%
\begin{equation}
\label{kttf} \IE T_F(V_1,\ldots,V_k) =
\IE T_F \bigl(\tau^{-1}(U_1),\ldots,\tau
^{-1}(U_k) \bigr) = t(F,\kappa_{\tau^{-1}})
\end{equation}
is immediate. Hence, Lemma~\ref{lem3} implies that the second term in
\eqref{tistep} approaches $0$ as $n \rightarrow\infty$.

We will use the main result of \cite{McDiarmid1989} to show that the
first term in \eqref{tistep} also vanishes. If $f$ is a function in $N$
arguments such that
changing the $i$th coordinate will change the value of $f$ by at
most $c_i$ and if $Y=(Y_1,\ldots,Y_N)$ are independent random
variables, then
%
\begin{equation}
\label{M89} \IP \bigl[\bigl\vert f(Y) - \IE f(Y)\bigr\vert \geq\eps \bigr]\leq2\exp
\biggl(-\frac
{2\eps
^2}{\sum_{i=1}^Nc_i^2} \biggr).
\end{equation}
Note now that, if $G$ is a graph with $n$ vertices, then $t(F,G)$
changes by at
most $\frac{k(k-1)}{n(n-1)}$ if one edge is changed. Applying
McDiarmid's concentration inequality to $t(F,G(x,\kappa))$ [with $f$
being a function of the $N={n\choose2}$ random edges], we therefore
have that, for every fixed $x\in[0,1]^n$,
%
\begin{equation}\qquad
\label{ci1term} \IP \bigl[ \bigl\vert t \bigl(F,G(x,\kappa) \bigr)-
\mu_F(x) \bigr\vert >\eps \bigr]  \leq2\exp \biggl(-\frac{2{\eps}^2}{{{n\choose
2}} ({k(k-1)}/{(n(n-1)}) )^2}
\biggr).
\end{equation}
Using Borel--Cantelli, we can conclude that
%
\begin{equation}
\label{bc1term} \bigl\vert t \bigl(F,G \bigl(X^{(n)},\kappa \bigr) \bigr) -
\mu_F \bigl(X^{(n)} \bigr) \to0 \bigr\vert
\end{equation}
almost surely as $n\toinf$. This proves the claim.
\end{pf*}

\begin{pf*}{Proof of Corollary~\ref{cor1}}
We can essentially imitate the proof of Theorem~\ref{THM1} to obtain
this result; the one difference being that we have to control
\[
\bigl\vert t \bigl(F,G \bigl(X^{(n)},H_n,\kappa \bigr) \bigr)-
\mu_F \bigl(X^{(n)} \bigr)\bigr\vert.
\]
We need to be bit
careful at (\ref{ci1term}) because of the dependencies introduced by $H_n$.
Suppose $E(H_n) = m_n \equiv m$. Applying McDiarmid's concentration
inequality to $t(F,G(x,H_n,\kappa))$ [with $f$ being a function of
the $N={n\choose2}-m$ random edges], we therefore have,
with $G=G(x,H_n,\kappa)$,
\begin{eqnarray*}
&&\IP \bigl[\bigl \vert t(F,G)-\mu_F(x) \bigr\vert >\eps \bigr]
\\
&&\qquad \leq \IP \bigl[ \bigl\vert t(F,G)-\IE t(F,G)\bigr \vert >\eps-
 \bigl\vert\mu
_F(x)-\IE t(F,G) \bigr\vert \bigr]
\\
&&\qquad \leq 2\exp \biggl(-\frac{2 (\eps-m{(k(k-1)}/{(n(n-1)))}
)^2}{ ({n\choose
2}-m ) ({k(k-1)}/{(n(n-1))} )^2} \biggr).
\end{eqnarray*}
As $m=\lito(n^2)$,
it follows from Borel--Cantelli that
\[
\bigl\vert t \bigl(F,G \bigl(X^{(n)},H_n,\kappa \bigr) \bigr) -
\mu_F \bigl(X^{(n)} \bigr)\bigr\vert \to0.
\]
\upqed\end{pf*}

\section{Applications}\label{sec4}

In this section, we will discuss two different sampling schemes, namely
a Markov chain model, where each respondent gives exactly one
referral, and a Poisson branching process model, where each respondent
gives a Poisson number of referrals (and thus, allowing for no
referrals). For both procedures, we essentially need to
establish \eqref{3}. Once this is done, Theorem~\ref{THM1} automatically
yields the corresponding convergence provided $\kappa$ is continuous.
We compare the two procedures
for a concrete parametrised standard kernel under different parameter
values.

In order to avoid that the standard kernel decomposes into two or more
disconnected parts, it is natural to impose an irreducibility
condition. We follow \cite{Janson2008a}.
Denote by $\Vol(A)$ the Lebesgue measure of $A\subset[0,1]$ and
let $A^c=[0,1]\setminus A$. We say that a standard kernel is \emph
{connected}, if $0<\Vol(A)<1$ implies
%
\begin{equation}
\label{6} \int_{A}\int_{A^c}
\kappa(x,y)\,dx\,dy>0.
\end{equation}
Loosely speaking, this condition guarantees that there can be links
from any set $A$ into its complement, so that no area can remain
disconnected from the rest of the graph [at least as $n\toinf$; for a
finite realisation of $G(n,\kappa)$, it may of course happen that the
graph consists of disconnected components]. Note that \eqref{6}
implies in
particular that
\[
\int_A\int_0^1
\kappa(x,y)\,dy > 0
\]
for all $A$ with $\Vol(A)>0$. This only guarantees that almost all $x$
have positive degree.
In order to avoid technicalities, we shall assume that \emph{all} $x$
have positive degree, that is,
%
\begin{equation}
\label{7} \int_0^1 \kappa(x,y)\,dy >0\qquad
\mbox{for all $x\in[0,1]$.}
\end{equation}
If \eqref{7} is satisfied, we say that a standard kernel is \emph{positive}.

\subsection{One-referral Markov chain sampling} \label{sec4.1}

The first model is a procedure where each respondent is asked (or
rather ``forced'') to give exactly one referral, resulting in one single
chain of referrals. We assume that these referrals happen in a
Markovian way, and a respondent of type $x$ chooses the referral
proportional to $\kappa(x,y)\,dy$. More rigorously, define the Markov kernel
%
\begin{equation}
\label{8} K_\kappa(x,dy):=\frac{ \kappa(x,y)\,dy }{\int_0^1 \kappa(x,v)\,dv}.
\end{equation}
Under \eqref{7}, the kernel is well defined.

\begin{proposition} Let $\kappa$ be a positive and connected standard
kernel, and let $X = (X_1,X_2,\ldots)$ be a Markov chain with Markov
kernel $K_\kappa$. Then $X$ has a unique invariant
probability measure $\pi$ given by
%
\begin{equation}
\label{9} \frac{\pi(dx)}{dx} = \frac{\int_0^1\kappa(x,v)\,dv}{\int_0^1\int_0^1
\kappa(u,v)\,du\,dv}.
\end{equation}
Furthermore, for every measurable and bounded function $f$ and for
almost every $x\in[0,1]$ we have
%
\begin{equation}
\label{10} \lim_{n\toinf}\frac{1}{n}\sum
_{i=1}^n f(X_i) = \int
_0^1 f(x) \pi(dx)
\end{equation}
$\IP[ \cdot|X_1=x]$-almost surely.
\end{proposition}

\begin{pf} Let us first prove that $K_\kappa$ does not have any invariant
measures with atoms. Assume that $\rho$ is an invariant measure. Write
$\rho= \rho^* + \rho'$, where $\rho^*$ is the atomic and $\rho'$ is
the nonatomic parts, and assume that
$\rho^*$ is not the zero measure. Let $A^*$ be the
support of $\rho^*$; note that $A^*$ is countable and that $\rho(A^*)>0$.
However, $K(x,A^*)=0$ for all $x\in[0,1]$ due to \eqref{8} and, therefore,
\[
\rho \bigl(A^* \bigr) = \int K_\kappa \bigl(x,A^* \bigr)\,d\rho(x) = 0,
\]
which is a contradiction.

We now use Yosida's ergodic decomposition to prove that $\pi$ is the only
invariant probability measure with respect to $K_k$ and that \eqref{10}
holds; see
\cite{Yosida1980} and \cite{HerandezLerma1998}.

Recall that an \emph{invariant set} is a set $A$ such that $K_\kappa
(x,A)=1$ for all
$x\in A$, that is,
%
\begin{equation}
\label{11} \int_A\kappa(x,y)\,dy = \int
_0^1\kappa(x,y)\,dy \qquad\mbox{for all $x\in A$}.
\end{equation}
Hence, we must have $\int_{A^c}\kappa(x,y)\,dy = 0$ for all $x\in A$.
This implies that
$\int_A\int_{A^c}\kappa(x,y)\,dy\,dx = 0$, which by symmetry of $\kappa$
and \eqref{6}, implies that $\Vol(A)=0$ or $\Vol(A)=1$. The
case $\Vol
(A)=0$ can be excluded since the right-hand side of \eqref{11} is positive
by \eqref{7}. By the definition of $\pi$, it follows that, for every such
invariant set $A$, we have $\pi(A)=1$. Therefore, $\pi$ is an
ergodic measure in the Yosida sense. Now, Lemma~4.2 of \cite
{HerandezLerma1998} implies that $\pi$ is unique on the invariant sets
up to $\pi$-null sets, but since $\pi$ cannot have any atoms, $\pi$ is
unique on $[0,1]$; \eqref{10} now follows from Theorem~6.1(b) of \cite
{HerandezLerma1998}.
\end{pf}

It is worthwhile mentioning that \eqref{10} holds even if the Markov chain
exhibits certain periodic behaviour. For example, if $\kappa$ is such
that the resulting graph is bipartite, the resulting Markov chain does
not converge to its stationary distribution, but it is still ergodic.

\subsection{A Poisson branching process model} \label{sec4.2}

Let us consider a continuous-time, multi-type Galton--Watson branching
process with type space $[0,1]$ as follows. A~particle of type $x\in
[0,1]$ is assumed to have a standard exponential lifetime and during
that time it will give birth to new particles of type $y$ at
rate $\lambda\kappa(x,y)\,dy$ for some $\lambda>0$ independently of all
else. Let $T_t$ be the random point measure on $[0,1]$ given by all
particles ever born up to and including time $t$. We denote by $\delta
_x$ the point unit measure at $x\in[0,1]$ and we write $\IP_x[ \cdot]
= \IP[ \cdot|T_0 = \delta_x]$. Note that $T_t[0,1]$ the total number
of points in $T_t$.

In order to push all arguments through as easily as possible, we will
not only assume that the standard kernel positive and connected, but
make the (most likely unnecessarily) strong assumption that
%
\begin{equation}
\label{12} \inf_{0\leq x,y\leq1}\kappa(x,y) > 0.
\end{equation}
It is clear that \eqref{12} implies that $\kappa$ is both, connected and
positive.

\begin{proposition} Let $\lambda>0$, let $\kappa$ be a standard kernel
satisfying \eqref{12}, and let $T_t$ be the resulting branching process.
Then there exists $\alpha^*$ and a unique probability measure $\pi$
on $[0,1]$ satisfying
%
\begin{equation}
\label{13} \frac{\pi(dx)}{dx} = \frac{\lambda}{1+\alpha^*} \int_0^1
\kappa (x,u)\pi(du).
\end{equation}
Furthermore, if $\alpha^*>0$, then $\IP_x[\vert T_t\vert \toinf] >
0$, and,
for any measurable and bounded function $f$ and for almost all $x\in[0,1]$,
%
\begin{equation}
\label{14} \lim_{t\toinf} \frac{1}{\vert T_t\vert }\int
_0^1 f(y) T_t(dy) = \int
_0^1 f(y) \pi(dy)
\end{equation}
$\IP_x[ \cdot| \vert T_t\vert \toinf]$-almost surely.
\end{proposition}

\begin{pf}
We follow the setup of \cite{Jagers1996}; see also \cite{Jagers1989}.
Define the \emph{reproduction kernel}
\[
\mu(x, dy\times dt) = e^{-t}\lambda\kappa(x,y)\,dy\,dt,
\]
which, loosely speaking, is the expected number of offspring of
type $y$ that a particle of type $x$, born at time $0$, produces at
time $t$ (the prefactor $e^{-t}$ is simply the probability that
the $x$-particle survives until time $t$). Furthermore, define the
transition kernel
\[
\hat\mu_\alpha(x,dy) = \int_{0}^\infty
e^{-\alpha t}\mu(x,dy\times dt) = \frac{\lambda}{1+\alpha} \kappa(x,y)\,dy.
\]
It is not difficult to see that \eqref{12} implies that the
kernel $\hat
\mu
_0$ is irreducible with respect to the Lebesgue measure on $[0,1]$
(cf. \cite{Nummelin1984}, Example~2.1(b), page~11). Hence, there is a
number $\alpha^*$ such that the kernel $\mu_\alpha$ has convergence
radius $1$ (cf. \cite{Niemi1986}, Proposition~2.1). The
parameter $\alpha^*$ is commonly called the \emph{Malthusian
parameter}. Moreover, \eqref{12} also implies that $\mu_{\alpha}$ is
recurrent (cf. \cite{Ney1987}, Lemma~2.3). Hence, there is a $\sigma
$-finite measure $\pi$ and a strictly positive function $h$ defined
on $[0,1]$ (cf. \cite{Jagers1996}, page~42 and \cite{Nummelin1984},
Theorem~5.1)
such that
%
\begin{equation}
\label{15} \int_0^1 \hat
\mu_{\alpha^*}(x,dy)\pi(dx) = \pi(dy)
\end{equation}
and
\[
\int_0^1 h(y)\hat\mu_{\alpha^*}(x,dy) =
h(x).
\]
Note that \eqref{15} is just \eqref{13}. It is also straightforward
to show
that $\mu$ is \emph{positive} recurrent (cf. \cite{Jagers1996},
page~43). Since
\[
h(x) = \int_0^1 h(y) \frac{\lambda}{1+\alpha^*}
\kappa(x,y)\,dy \geq\frac{\lambda}{1+\alpha^*}\inf_{x,y} \kappa(x,y) \int
_0^1 h(y)\,dy,
\]
it is clear that $\inf h(x) >0$. This implies that $\pi$ is finite an
can be normed to a probability measure (cf. \cite{Jagers1996}, page~43)
and $h$ can be chosen so that $\int h(x)\pi(dx)=1$. Finally, it is
clear that $\mu$ is nonlattice and that there is $\eps>0$ such that
\[
\sup_x \mu \bigl(x,[0,1]\times[0,\eps] \bigr) < 1.
\]
These conditions are summarised as $\mu$ being \emph{nonlattice and
strictly Malthusian}.

Note that, since $\kappa\leq1$,  $\vert T_t\vert $ can be dominated
by a
unitype branching process where each particle has standard exponential
lifetime and produces offspring at rate $\lambda$. Therefore, for
fixed $t$, $\vert T_t\vert $ is uniformly integrable in the type of the
starting particle. Moreover, the usual ``$x\log x$'' condition follows
easily from the fact that the dominating branching process has finite
variance. Applying \cite{Jagers1996}, Theorem~2, it follows that, for
almost all $x$ and for $A\subset[0,1]$,
%
\begin{equation}
\label{16} e^{-\alpha^* t} T_t(A) \to\frac{\pi(A)}{\alpha^*\beta} W
\end{equation}
$\IP_x$-almost surely for some nonnegative random variable $W$ that
satisfies $\IE_x W = h(x)$, and for some $\beta$ (which is explicit,
but not of interest here). Note that clearly $\{W>0\}\subset\{\vert
T_t\vert \toinf\}$, but it is not immediate that the two sets are equal.
In order to make statements about \eqref{16} with $e^{-\alpha^* t}$
replaced by $1/\vert T_t\vert $, we need that
%
\begin{equation}
\label{17} \inf_x \IP_x[W>0] > 0,
\end{equation}
which guarantees that $\{W>0\}=\{\vert T_t\vert \toinf\}$ by \cite{Jagers1996},
Lemma~1.

In order to prove \eqref{17}, note that there must be a set $A$
with $\Vol
(A)>0$ such that $p_A := \inf_{x\in A}\IP_x[W>0] > 0$, for otherwise we
would have $\IP_x[W>0] = 0$ for almost all $x$ which is in
contradiction to $\IE_x W = h(x) > 0$ for almost all $x$. Let $M =
\inf_{x,y} \kappa(x,y)$, which by \eqref{12} is positive, and
\[
E_A = \{\mbox{1st particle has exactly one child of some type $y\in
A$}\}.
\]
Now,
\begin{eqnarray*}
\IP_x[W>0] & \geq&\IP_x[W>0, E_A]
\\
& =& \IP_x[W>0 | E_A] \IP_x[E_A]
\\
& \geq& p_A \int_0^\infty
e^{-t}\lambda M \Vol(A) t e^{-\lambda M \Vol
(A) t} \,dt,
\end{eqnarray*}
which is a positive lower bound independent of $x$. Hence, \eqref{17}
follows. From~\cite{Jagers1996}, Corollary~4, we have for almost
all $x$ that
\[
\frac{T_t(A)}{\vert T_t\vert } \to\pi(A)
\]
%
$\IP_x[ \cdot| \vert T_t\vert \toinf]$-almost surely for any
measurable $A\subset[0,1]$. Since $\pi$ is finite, it is easy to
extend this to \eqref{14} for bounded $f$.
\end{pf}

\subsection{A concrete standard kernel}\label{sec4.3}
In this section, we consider a particular standard kernel $\kappa
:[0,1]^2 \rightarrow[0,1]$ given as
%
\begin{equation}
\label{18} \kappa(x,y) = \cases{ \alpha,& \quad
$\mbox{if $0 \leq x \leq\gamma$ and
$0 \leq y \leq\gamma $,}$\vspace*{2pt}
\cr
\beta,& \quad $\mbox{if $\gamma< x \leq1$ and
$\gamma< y \leq1$,}$ \vspace*{2pt}
\cr
\delta,&\quad $\mbox{otherwise}$,}
\end{equation}
where $0 < \alpha, \gamma,\delta, \beta< 1 $.

One could think of $\kappa$ as a graph between two groups of vertices.
The internal connections between
a primary group A (say) are specified by $\alpha$ and a secondary
group B (specified) by $\beta$. The inter-connections
between the groups of vertices are specified by $\delta$. If we sample
the vertices ergodically with invariant measure $\pi$,
then Theorem~\ref{THM1} specifies that our limit graph will be
governed by
\[
\kappa_{\tau^{-1}} (x,y) = \cases{ \alpha,& \quad $\mbox{if $0 \leq x \leq\tau(
\gamma)$ and $0 \leq y \leq \tau (\gamma)$,}$\vspace*{2pt}
\cr
\beta,& \quad $\mbox{if $
\tau(\gamma) < x \leq1$ and $\tau(\gamma) < y \leq 1$,}$ \vspace*{2pt}
\cr
\delta,& \quad
$\mbox{otherwise,}$ }
\]
where $\tau(x) = \pi([0,x])$; see Figure~\ref{fig1} for a graphical
representation of the distortion in $\kappa$.

\begin{figure}

\includegraphics{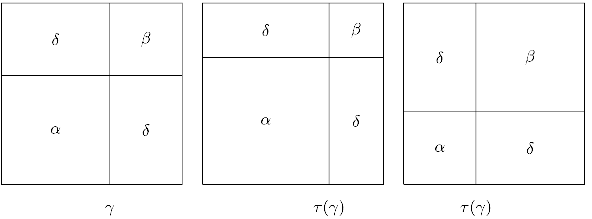}

\caption{Above is a pictorial representation of $\kappa$ (left), $\kappa_{
\tau^{-1}}$
when $\tau(\gamma) > \gamma$ (middle), and $\kappa_{\tau^{-1}}$
when $\tau(\gamma) < \gamma$ (right).}\label{fig1}
\end{figure}

%

We shall now compare $\kappa_{\tau^{-1}}$ in the sampling procedures
discussed in Sections \ref{sec4.1} and \ref{sec4.2}. In
the procedure discussed in Section~\ref{sec4.1}, we have $\pi$ given
by \eqref{9} and a routine calculation gives us that
the value of the distortion, denoted by $\tau_M$, at $\gamma$ is
given by
%
\begin{equation}
\label{19} \tau_M(\gamma) = {\frac{  ( \alpha\gamma-\delta\gamma
+\delta
 ) \gamma}{
\alpha{\gamma}^{2}-2 {\gamma}^{2}\delta+\beta{\gamma}^{2}+2
\delta
\gamma
-2 \beta\gamma+\beta}.}
\end{equation}
In the procedure discussed in Section~\ref{sec2}, we need to find $\pi
(dx)/dx = \nu(x)$, which satisfies
\[
\nu(x) = \frac{\lambda}{1+\alpha^*} \int_0^1
\kappa(x,y)\nu (y)\,dy, \qquad \int_0^1 \nu(x)\,dx = 1.
\]
In the case of \eqref{18}, this is equivalent to finding the largest
eigenvalue and corresponding eigenvector of a $(2\times2)$-matrix.
A standard calculation then shows that the value of the distortion,
denoted by $\tau_P$, at $\gamma$ is given by
%
\begin{equation}
\label{20} \tau_P(\gamma) = \frac{(\alpha+ \beta)\gamma-\beta+ s }{2\delta+ \gamma(\alpha
-2\delta+\beta) -\beta+s},
\end{equation}
where
\[
s = \sqrt{{\gamma}^{2} \bigl({(\alpha+\beta)}^{2} -4
\delta^{2} \bigr) +2\gamma \bigl(2 {\delta}^{2} -\alpha
\beta-{\beta}^{2} \bigr)+{\beta}^{2}}.
\]

In general, the formulae \eqref{19} and \eqref{20} do not compare in an
obvious manner with themselves or
with the unbiased sampling [$\tau(\gamma) = \gamma$]. In the
one-referral Markov chain sampling model,
a new vertex is chosen proportional to the values of $\kappa$, with the
proportionality constant being the
volume measure under $\kappa$. In contrast, in the Poisson branching
process model, due to the branching effect, the
offspring of a vertex will be from the regions governed by the
sectional area of $\kappa$ at the vertex. Thus, it
is natural to expect differences in bias between the two procedures. We
illustrate this via
three examples of $\alpha, \beta, \delta$ to illustrate the differences
in distortion between the two sampling procedures.

The first example we consider is when $\alpha=1/5$, $\delta
=1/200$, $\beta=1/5$. In Figure~\ref{fig2}, we plot $\tau(\gamma)$
as a
function of $\gamma$. One can quickly observe that for $\gamma= 0.5$
there is no distortion
in either sampling scheme as expected with $\tau_M(0.5) = \tau_P(0.5) =
0.5$. One observes that when $\gamma<0.5$ then $\tau_P(\gamma) <
\tau_M
(\gamma) < \gamma$ and when $\gamma>0.5$ then $ \gamma< \tau
_M(\gamma)
< \tau_P (\gamma)$. This indicates
that the Poisson branching process model will result in a larger bias
toward the larger group (the secondary group B when $\gamma< 0.5$
and the primary group A when $\gamma>0.5$). This is expected as both
the primary group A and secondary group B are
similarly well connected internally, but a small $\delta$ implies that
they are poorly interconnected.

\begin{figure}

\includegraphics{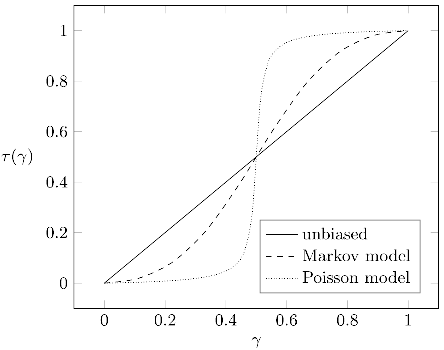}

\caption{Distortion plot for two large groups that are
internally well connected, but where there are not many connections
between the groups ($\alpha=1/5$, $\delta=1/200$, $\beta=1/5$).}\label{fig2}
\end{figure}

%
%

%

\begin{figure}[b]

\includegraphics{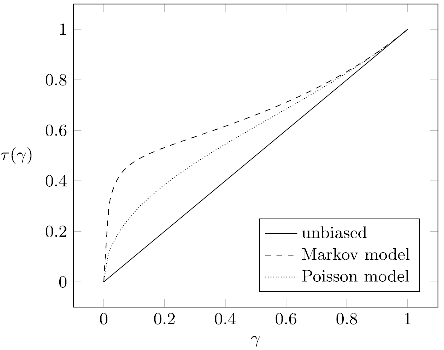}

\caption{Distortion plot for a graph where the primary
group is well connected to itself and to a secondary group, but where
the secondary group is not well connected within itself ($\alpha
=1/5$, $\delta=1/5$, $\beta=1/200$).} \label{fig3}
\end{figure}

In the second example (see Figure~\ref{fig3}), we consider $\alpha
=1/5$, $\delta=1/5$ and $\beta=1/200$. Group A has a fair number of
connections within itself, and there are fair number of connections
between the Groups A and B, but with a small $\beta$, Group B has a
smaller number of connections within itself. Note that for all $0 <
\gamma< 1$, $\gamma< \tau_P(\gamma) < \tau_M(\gamma)$, indicating a
larger bias in the one-referral Markov chain sampling procedure.

%

Finally, we consider the case $\alpha=1/5$, $\delta=1/200$, $\beta
=1/200$ (see Figure~\ref{fig4}). With interconnection probability and
within Group B
connection probabilities being small this time there is a strong bias
toward selecting vertices from the primary Group A.
In plot shown in Figure~\ref{fig4}, we can see that the bias is more
pronounced this time in the Poisson branching process sampling
procedure.

\begin{figure}

\includegraphics{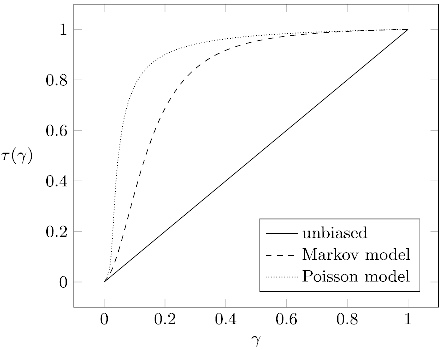}

\caption{Distortion plot for a graph where the primary
group is well connected to itself,
and where the secondary group is neither well connected to the primary
group nor within itself ($\alpha=1/5$, $\delta=1/200$, $\beta=1/200$).}\label{fig4}
\end{figure}

%

In conclusion, depending on the size, connectedness of the groups and
interconnections between them,
the sampling scheme has to be chosen appropriately to control the bias.

\section*{Acknowledgements}

We thank Andrew Barbour and Don Dawson for helpful discussions, and
we are grateful to the anonymous referee for making
suggestions that resulted in improving the paper and for pointing
out an important error in the proof of the main theorem; in particular, this
resulted in the introduction of
assumption (ii) in Theorem~\ref{THM1}.

Parts of this work was done
when S. Athreya visited the Institute of Mathematical Sciences, NUS. The
authors would like to express their gratitude for the kind
hospitality.



%





\printaddresses
\end{document}